\documentclass[aap,MSNbibl,seceqn,nameyear,dvips]{arximspdf}
\usepackage{mathrsfs}
\usepackage{accents}
\usepackage{mathbh}

\doi{10.1214/13-AAP929} %kopijuoti is PTS
\volume{24}
\issue{1}
\pubyear{2014}
\firstpage{422}
\lastpage{446}

\makeatletter
\newcommand{\rrvert}{\vert}
\newcommand{\llvert}{\vert}

\newcommand{\A}{\mathscr{A}}
\newcommand{\B}{\mathscr{B}}
\newcommand{\C}{\mathscr{C}}
\newcommand{\R}{\mathbb{R}}
\newcommand{\Z}{\mathbb{Z}}
\renewcommand{\P}{\mathbb{P}}
\newcommand{\E}{\mathbb{E}}
\newcommand{\shift}{\mathscr{S}}
\newcommand{\ones}{\vec{1}}
\newcommand{\vtheta}{\vec{\theta}}
\newcommand{\vpsi}{\vec{\psi}}
\newcommand{\vxi}{\vec{\xi}}
\newcommand{\given}{|}
\newcommand{\un}[1]{\underaccent{\bar}{#1}}

\newtheorem{theorem}{Theorem}[section]

\newtheorem{lemma}[theorem]{Lemma}
\newtheorem{corollary}[theorem]{Corollary}
\newtheorem{proposition}[theorem]{Proposition}

\newproclaim{remark}{Remark}

\makeatother

\begin{document}
\begin{frontmatter}

\title{Loss of memory of hidden Markov models and Lyapunov
exponents\thanksref{T0}}
\runtitle{Loss of memory of hidden Markov models}

\thankstext{T0}{Supported in part by FAPESPs project
\textit{Consistent estimation of stochastic processes
with variable length memory} (2009/09411-8),
CNPqs projects \textit{Stochastic systems with variable length interactions}
(476501/2009-1) and USP project
\textit{Mathematics}, \textit{computation}, \textit{language and the brain}
(11.1.9367.1.5).}

\begin{aug}
\author[A]{\fnms{Pierre} \snm{Collet}\thanksref{t2}\ead[label=e1]{pierre.collet@cpht.polytechnique.fr}}
\and
\author[B]{\fnms{Florencia} \snm{Leonardi}\corref{}\thanksref{t1}\ead[label=e2]{florencia@usp.br}}
\runauthor{P. Collet and F. Leonardi}
\affiliation{Ecole Polytechnique and Universidade de S\~ao Paulo}
\address[A]{Centre de Physique Th\'eorique\\
CNRS UMR 7644\\
Ecole Polytechnique\\
91128 Palaiseau Cedex\\
France \\
\printead{e1}}
\address[B]{Instituto de Matem\'atica e Estat\'{i}stica\\
Universidade de S\~ao Paulo\\
CEP 05508-090, Cidade Universit\'aria\\
S\~ao Paulo, SP\\
Brasil\\
\printead{e2}}
\end{aug}

\thankstext{t2}{Supported in part by the
DynEurBraz European project.}
\thankstext{t1}{Supported in part by a CNPq-Brazil fellowship.}

% HISTORY:
\received{\smonth{11} \syear{2011}}

% ABSTRACT
%
\begin{abstract}
In this paper we prove that the
asymptotic rate of exponential loss of memory of a
finite state hidden Markov model is bounded above by the
difference of the first two Lyapunov exponents of a certain product of
matrices. We also show that this bound is in fact realized, namely for
almost all realizations of the observed process we can find symbols
where the asymptotic exponential rate of loss of memory attains the
difference of the first two Lyapunov exponents.
These results are derived in particular for the observed process and
for the filter; that is, for the distribution
of the hidden state conditioned on the observed sequence.
We also prove similar results in total variation.
\end{abstract}

% KEYWORDS

\end{frontmatter}

%%%%%%%%%%%
%s1 #&#
\section{Introduction}
%%%%%%%%%%%

Let $(X_t)_{t\in\Z}$ be a Markov chain over a finite alphabet $\A$. We
consider a probabilistic function $(Z_t)_{t\in\Z}$ of this chain, a
model introduced by \citet{petrie1969}. More precisely, there is
another finite alphabet $\B$
and for any $X_{t}$ we
choose at random a $Z_{t}$ in $\B$.
The random choice of $Z_{t}$
depends only on
the value $X_{t}$ of the original process at time $t$.
The process $(Z_t)$ is the observed process and $(X_t)$ is the hidden
process. This model is called a hidden Markov process.

We are interested in the asymptotic loss of memory of the processes
$(X_t)_{t\in\Z}$ and $(Z_t)_{t\in\Z}$ conditioned on the observed
sequence. For example, if the conditional probability of $Z_{t}$ given
$X_{t}$ does not depend on $X_{t}$, the process $(Z_{t})_{t\in\Z}$ is
an independent process. Another trivial example is when there is no
random choice, namely $Z_{t}=X_{t}$, in this case the process
$(Z_{t})_{t\in\Z}$ is Markovian. However, as we will see, under natural
assumptions, the process $(Z_{t})_{t\in\Z}$ has infinite memory. On the
other hand, a particularly interesting question from the point of view
of applications is to consider the loss of memory of the filter; that
is, the distribution of $X_0$ conditioned on the past observed sequence
$Z_{-1},\ldots,Z_{-n+1}$ and for different initial conditions on
$X_{-n}$; see, for example, \citet{cappe2005}.

Our goal is to investigate how fast these processes loose memory.

Exponential upper bounds for this asymptotic loss of memory have been
obtained in various papers; see, for example, \citet{moulines1,moulines2}
and references therein. For the case of projections of Markov chains
and the relation with Gibbs measures, see \citet{chazottes2} and
references therein.

In the present paper, under generic assumptions, we prove that the
asymptotic rate of exponential loss of memory is bounded above by the
difference of the first two Lyapunov exponents of a certain product of
matrices. We also show that this bound is in fact realized, namely for
almost all realizations of the process $(Z_{t})_{t\in\Z}$, we can find
symbols where the asymptotic exponential rate of loss of memory
attains the difference of the first two Lyapunov exponents.
As far as we know our results provide the first lower
bounds for the loss of memory of these processes. Similar results (in
particular lower bounds) are also obtained in the total variation
distance.

Conditioned on the observed sequence $Z_{-1},\ldots, Z_{-n+1}$, we
have considered different possibilities for the initial
distribution at time $-n$, namely one can either give the initial
distribution of $X_{-n}$
or the initial distribution of $Z_{-n}$. Similarly one can ask for the
distribution of $X_0$
(the hidden present state) or of $Z_0$ (the observable present state).

As an application, we consider the case of a randomly perturbed Markov
chain with two symbols. We show that the asymptotic rate of loss of
memory can be expanded in powers of the perturbation with a logarithmic
singularity. This was our original motivation coming from our previous
work with Galves [\citet{collet}].

The relation between product of random matrices and hidden Markov
models was previously described in \citet{JSS}. In this paper it was
proved in particular that the first Lyapunov exponent is the opposite
of the entropy of the process.

The content of the paper is as follows. In Section \ref{main} we give a
precise definition of the asymptotic exponential rate of
loss of memory and state the main results about the relation of this
rate with the first two Lyapunov exponents.

Proofs are given in Section \ref{proofs}. They rely on more general
propositions which allow to treat at once the different situations of
initial distributions and present distributions.
In Section \ref{binary} we give the application to the
random perturbation of a two states Markov chain.

%%%%%%%%%%%%%%%%%%%%%%%%%
%s2 #&#
\section{Definitions and main results}\label{main}
%%%%%%%%%%%%%%%%%%%%%%%%%

Let $(X_t)_{t\in\Z}$ be an irreducible aperiodic Markov chain over a
finite alphabet $\A$ with transition probability matrix
$p(\cdot|\cdot)$ and unique invariant measure $\pi$. Without loss of
generality we will assume $\A=\{1,2,\ldots,k\}$. In the sequel, we
will use the shorthand notation $x_{r}^{s}$ for a sequence of symbols
$(x_{r},\ldots,x_{s})$ ($r\le s$). Consider another finite
alphabet $\B=\{1,2,\ldots,\ell\}$,
and a process
$(Z_t)_{t\in\Z}$, a probabilistic function of the Markov chain
$(X_t)_{t\in\Z}$ over~$\B$. That is, there exists a matrix
$q(\cdot|\cdot) \in\R^{k\times\ell}$ such that for any $n\geq0$,
any $z_0^n\in\B^{n+1}$ and any $x_0^n\in\A^{n+1}$, we have
%
%e2.1 #&#
\begin{equation}
\P\bigl(Z_0^n=z_0^n|X_0^n=x_0^n
\bigr) = \prod_{i=0}^n
\P(Z_i=z_i|X_i=x_i) = \prod
_{i=0}^n q(z_i|x_i).
\end{equation}

From now on, the symbol $\un{z}$ will represent an element in
$\B^\Z$. Define the shift-operator $\shift\dvtx\B^\Z\to\B^\Z$
by
\[
(\shift\un{z})_n = z_{n+1}.
\]
The shift is invertible, and its inverse is given by
\[
\bigl(\shift^{-1}\un{z} \bigr)_n = z_{n-1}.
\]

To state our results we will need the following hypotheses:
\begin{longlist}[(H3)]
\item[(H1)] $\min_{i,j} p(j|i) > 0$, $\min_{i,m} q(m|i) > 0$.
\item[(H2)] $\det(p) \neq0$.
\item[(H3)] $\operatorname{rank}(q) = k$.
\end{longlist}
Note that hypothesis (H3) implies $\ell\geq k$.

For the convenience of the reader we recall Oseledec's theorem
in finite dimension; see, for example, \citet{ledrappier,katok1995}.
As usual, we denote by
$\log^+(x) = \max(\log(x),0)$.

\begin{oseledec*}
Let $(\Omega,\mu)$ be a probability space and let $T$ be a measurable
transformation of $\Omega$
such that $\mu$ is $T$-ergodic. Let $L_{\omega}$
be a measurable function from $\Omega$ to $\mathscr{L}(\R^k)$ (the
space of linear operators of $\R^k$ into itself). Assume the function
$L_{\omega}$ satisfies
\[
\int\log^+ \| L_{\omega}\| \,d\mu(\omega) < +\infty.
\]
Then, there exist $\lambda_1 > \lambda_2 > \cdots> \lambda_s$, with
$s\leq k$ and there exists an invariant set $\tilde\Omega\subset
\Omega$
of full measure $(\mu(\Omega\setminus\tilde\Omega) = 0)$ such that for
all $\omega\in\tilde\Omega$ there exist $s+1$ sub-vector spaces
\[
\R^k = V^{(1)}_\omega\supsetneq
V^{(2)}_\omega\supsetneq\cdots\supsetneq V^{(s+1)}_\omega=
\{\vec{0}\}
\]
such that for any $\vec{v}\in V^{(j)}_\omega\setminus
V^{(j+1)}_\omega$ ($1\leq j\leq s$) we have
\[
\lim_{n\to+\infty} \frac1n \log\bigl\| L_{\omega}^{[n]}
\vec{v}\bigr\|= \lambda_j,
\]
where $L_\omega^{[n]} = L_{T^{n-1}(\omega)} \cdots L_\omega$.
Moreover, the subspaces satisfy the relation
\[
L_\omega V^{(j)}_\omega\subseteq
V^{(j)}_{T\omega}.
\]
\end{oseledec*}

The numbers $\lambda_1, \lambda_2,\ldots,\lambda_s$ are called the
Lyapunov exponents.

In the sequel we will use this theorem with $\Omega= \B^\Z$, $\mu$
the stationary ergodic measure
of the process $(Z_t)_{t\in\Z}$ [\citet{cappe2005}], $T = \shift
^{-1}$ and
$L_{\un{z}}$ the linear operator in $\R^k$ with matrix given by
\[
(L_{\un{z}})_{i,j} = q(z_0|i)p(j|i).
\]
With this notation, we have, for example,
\[
\P\bigl(X_0=a,Z_{-n+1}^{-1}=z_{-n+1}^{-1},X_{-n}=b
\bigr)= \bigl\langle\vtheta_b, L^{[n-1]}_{\shift^{-1}\un{z}}
\ones_a\bigr\rangle,
\]
where $(\vtheta_b)_i = p(i|b)$ and $\ones_a$ is the basis vector with
component number $a$ equal to one.

From now on we will use the $\ell^{2}$ norm $\|\cdot\|$ and the
corresponding scalar product on $\R^{k}$.
Note that from our definition of $ L_{\un{z}}$ we have
\[
\sup_{\un{z}} \| L_{\un{z}} \|< +\infty.
\]
Therefore we can apply Oseledec's theorem to get the existence
of the Lyapunov exponents.

For any $\un{z}\in\B^\Z$, for probabilities $\rho$ on $\A$, $\eta
$ on $\B$
and any
integer $n$, we define two probabilities on $\A$
by
\[
\nu^{[n]}_{\un{z}, \rho}(a) =\frac{
\sum_{b\in\A}\P(X_0=a,
Z_{-n+1}^{-1}=z_{-n+1}^{-1},X_{-n}=b)\rho(b)} {
\sum_{b\in\A}\P(Z_{-n+1}^{-1}=z_{-n+1}^{-1},X_{-n}=b)\rho(b)},\qquad a\in\A,
\]
and
\[
\sigma^{[n]}_{\un{z}, \eta}(a) =\frac{
\sum_{c\in\B}\P(X_0=a,
Z_{-n+1}^{-1}=z_{-n+1}^{-1},Z_{-n}=c)\eta(c)} {
\sum_{c\in\B}\P(Z_{-n+1}^{-1}=z_{-n+1}^{-1},Z_{-n}=c)\eta(c)},\qquad a\in\A. %
\]

These are the probabilities of $X_{0}$ conditioned on the observed string
$z_{-n+1}^{-1}$ when the distribution of $X_{-n}$ is $\rho$
(resp., the distribution of $Z_{-n}$ is $\eta$).

When $\rho$ is a Dirac measure concentrated on $b$ we will simply denote
the measure $\nu^{[n]}_{\un{z}, \rho}$ by $\nu^{[n]}_{\un{z}, b}$, and
similarly for
$\sigma^{[n]}_{\un{z}, \eta}$.

We can state now our main results.

%th2.1 #&#
\begin{theorem}\label{upperbound}
Under the hypothesis \textup{(H1)}, for each $a\in\A$, for any probabilities
$\rho$
and $\rho'$ on $\A$,
\[
\limsup_{n\to+\infty} \frac1n \log\bigl\llvert\nu^{[n]}_{\un{z}, \rho}(a)-
\nu^{[n]}_{\un{z}, \rho'}(a) \bigr\rrvert\leq\lambda_2-
\lambda_1,
\]
$\mu$-almost surely. Similarly, under the hypothesis \textup{(H1)},
for each $a\in\A$, for any probabilities $\eta$
and $\eta'$ on $\B$,
\[
\limsup_{n\to+\infty} \frac1n \log\bigl\llvert\sigma^{[n]}_{\un{z},
\eta}(a)-
\sigma^{[n]}_{\un{z}, \eta'}(a) \bigr\rrvert\leq\lambda_2-
\lambda_1,
\]
$\mu$-almost surely.
\end{theorem}

\begin{remark*}
When $\A=\B$ and $q$ is the identity matrix, $(Z_{t})_{t\in\Z
}=(X_{t})_{t\in\Z}$ is a
Markov chain. The second part of hypothesis (H1) does not hold, but it
is easy to adapt the proof of Theorem \ref{upperbound} for this
particular case. It is easy to verify recursively that the matrices
$L_{\un{z}}^{[n]}$ are of rank one. The Lyapunov exponents can be
computed explicitly. One gets $\lambda_{1}=-H(p)$ (the entropy of the
Markov chain with transition probability $p$) from the ergodic
theorem, and $\lambda_{2}=-\infty$ with multiplicity $k-1$.
\end{remark*}

%th2.2 #&#
\begin{theorem}\label{lowerbound}
Under hypotheses \textup{(H1)--(H2)}, for $\mu$-almost all $\un{z}$
there exists $a,b,c\in\A$ (which may depend on $\un{z}$) such that
\[
\limsup_{n\to+\infty} \frac1n \log\bigl\llvert\nu^{[n]}_{\un{z}, b}(a)-
\nu^{[n]}_{\un{z}, c}(a) \bigr\rrvert= \lambda_2-
\lambda_1.
\]
Under hypotheses \textup{(H1)--(H3)}, for $\mu$-almost all $\un{z}$
there exists $a\in\A$, $b, c\in\B$ (which may depend on $\un{z}$)
such that
\[
\limsup_{n\to+\infty} \frac1n \log\bigl\llvert\sigma^{[n]}_{\un{z}, b}(a)-
\sigma^{[n]}_{\un{z}, c}(a) \bigr\rrvert= \lambda_2-
\lambda_1.
\]
\end{theorem}

As a corollary, we derive equivalent results for the loss of memory of
the process $(Z_{t})_{t\in\Z}$.
For any $\un{z}\in\B^\Z$, for probabilities $\rho$ on $\A$, $\eta
$ on $\B$
and any
integer $n$, we define two probabilities on $\B$
by
\[
\tilde\nu^{[n]}_{\un{z}, \rho}(e) =\frac{
\sum_{b\in\A}\P(Z_0=e,
Z_{-n+1}^{-1}=z_{-n+1}^{-1},X_{-n}=b)\rho(b)} {
\sum_{b\in\A}\P(Z_{-n+1}^{-1}=z_{-n+1}^{-1},X_{-n}=b)\rho(b)},\qquad e\in\B,
\]
and
\[
\tilde\sigma^{[n]}_{\un{z}, \eta}(e) = \frac{\sum_{c\in\B}\P
(Z_0=e,Z_{-n+1}^{-1}=z_{-n+1}^{-1},Z_{-n}=c)\eta(c)} {
\sum_{c\in\B}\P(Z_{-n+1}^{-1}=z_{-n+1}^{-1},Z_{-n}=c)\eta(c)},\qquad e\in\B.
\]
%
%co2.3 #&#
\begin{corollary}\label{cormain}
Under the hypothesis \textup{(H1)}, for each $e\in\B$, for any probabilities
$\rho$
and $\rho'$ on $\A$,
\[
\limsup_{n\to+\infty} \frac1n \log\bigl\llvert\tilde
\nu^{[n]}_{\un{z}, \rho}(e)-\tilde\nu^{[n]}_{\un{z},
\rho'}(e)
\bigr\rrvert\leq\lambda_2-\lambda_1,
\]
$\mu$-almost surely. Similarly, under the hypothesis \textup{(H1)},
for each $e\in\B$, for any probabilities $\eta$
and $\eta'$ on $\B$,
\[
\limsup_{n\to+\infty} \frac1n \log\bigl\llvert\tilde
\sigma^{[n]}_{\un{z}, \eta}(e)-\tilde\sigma^{[n]}_{\un
{z}, \eta'}(e)
\bigr\rrvert\leq\lambda_2-\lambda_1,
\]
$\mu$-almost surely.

Moreover, under hypotheses \textup{(H1)--(H3)}, for $\mu$-almost all $\un{z}$
there exists $e\in\B$, $b, c \in\A$ (which may depend on
$\un{z}$) such that
\[
\limsup_{n\to+\infty} \frac1n \log\bigl\llvert\tilde
\nu^{[n]}_{\un{z}, b}(e)-\tilde\nu^{[n]}_{\un{z}, c}(e)
\bigr\rrvert= \lambda_2-\lambda_1.
\]
Under hypotheses \textup{(H1)--(H3)}, for $\mu$-almost all $\un{z}$
there exists $e,b, c\in\B$ (which may depend on $\un{z}$) such that
\[
\limsup_{n\to+\infty} \frac1n \log\bigl\llvert\tilde
\sigma^{[n]}_{\un{z}, b}(e)-\tilde\sigma^{[n]}_{\un{z}, c}(e)
\bigr\rrvert= \lambda_2-\lambda_1.
\]
\end{corollary}

From a practical point of view, one can prove various lower bounds for
the quantity $\lambda_{2}-\lambda_{1}$. As an example we give the
following result.

Let
\[
\Gamma= \frac{1}{\min_{m,i}\{q(m|i)\}}. %
\]

%pr2.4 #&#
\begin{proposition}\label{diff}
Under hypotheses \textup{(H1)--(H2)} we have
\[
\lambda_2-\lambda_1 \ge\frac{1}{k-1} \log\bigl|\det(p)
\bigr| -\frac{k}{k-1}\log\Gamma. %
\]
\end{proposition}

We now state a related result using the total variation distance
between the distributions. This result will
only use a weaker version of hypothesis (H3), namely

(H3$'$), there exists $b,c\in\B$ in
and $i\in\A$ such that
%
%e2.2 #&#
\begin{equation}
\label{pasegaux} q(b|i)\neq q(c|i).
\end{equation}
Note that under this hypothesis, we do not assume any relation between
the cardinality of $\A$ and the cardinality of $\B$ (we require of
course the cardinality of $\B$ being at least two).

We recall that the total variation distance $\operatorname{TV}(\nu_{1},\nu_{2})$
between two measures $\nu_{1}$ and
$\nu_{2}$ on $\A$ is defined by
\[
\operatorname{TV}(\nu_{1},\nu_{2})=\frac{1}{2} \sum
_{a\in\A} \bigl|\nu_{1}(a)-\nu_{2}(a) \bigr|.
\]
Similar definitions are given for two measures on $\B$.

It follows at once that under the hypothesis (H1) we have $\mu$-almost surely
\[
\limsup_{n\to+\infty} \frac1n \log \operatorname{TV} \bigl( \nu^{[n]}_{\un{z}, \rho}(a)-
\nu^{[n]}_{\un{z}, \rho'}(a) \bigr) \leq\lambda_2-
\lambda_1,
\]
and similarly for the measures $\sigma$, $\tilde\nu$ and
$\tilde\sigma$. In order to state a lower bound for these quantities,
we need to recall a result about Lyapunov dimensions.

We denote by $s$ the number of different Lyapunov exponents, and by
$m_{i}$ ($1\le i\le s$)
the multiplicity of the exponent $\lambda_{i}$, namely
\[
m_{i}=\operatorname{dim}\bigl(V^{(i)}_\omega\bigr) -
\operatorname{dim}\bigl(V^{(i+1)}_\omega\bigr). %
\]
It follows from Oseledec's theorem that these numbers are $\mu$-almost
surely constant.

%th2.5 #&#
\begin{theorem}\label{tv}
Assume hypotheses \textup{(H1)--(H2)}. Then for $\mu$-almost every
$\un{z}$ and for any pair $(b,c)$ of elements in $\B$ satisfying
\textup{(H3$'$)} we have
\begin{eqnarray*}
\liminf_{n\to\infty}\frac{1}{n}\log \operatorname{TV} \bigl(
\sigma^{[n]}_{\un{z},b},\sigma^{[n]}_{\un{z},c} \bigr)
&\ge& \lambda_{s}-\lambda_{1}+\sum
_{i=2}^{s}m_{i} (\lambda_{i}-
\lambda_{2} )
\\
&\ge& 2 \bigl({\log}|{\det p}|-k\log\Gamma\bigr).
\end{eqnarray*}
\end{theorem}

\begin{remark*}
Similar lower bounds for the total variation distance between the measures
$\nu^{[n]}_{\un{z},b}$ and $\nu^{[n]}_{\un{z},c}$ can be proven
under hypotheses
(H1)--(H2). In the case of the total variation distance
between $\tilde\sigma^{[n]}_{\un{z},b}$ and $\tilde\sigma
^{[n]}_{\un{z},c}$ (resp.,
between $\tilde\nu^{[n]}_{\un{z},b}$ and $\tilde\nu^{[n]}_{\un
{z},c}$) we can prove also the same lower
bounds, but this requires the full set of hypotheses (H1)--(H3).
\end{remark*}

%%%%%%%%%%%%%%%
%s3 #&#
\section{Proofs}\label{proofs}
%%%%%%%%%%%%%%%

We begin by proving some lemmas which will be useful later. We
introduce the order $(\R^k,\leq)$ given by $\vec{v}\leq\vec{w}$ if and
only if $v_i\leq w_i$ for all $i=1,\ldots,k$. When needed, we will also
make use of the symbols $<$, $>$ and $\geq$, defined in an analogous
way. Note that since the matrices $L_{\un{z}}$ have strictly positive
entries, if $\vec{v}\leq\vec{w}$, then $L_{\un{z}}\vec{v}\leq
L_{\un{z}}\vec{w}$. We will use the notation $\ones\in\R^k$ for the
vector with components $(\ones)_i = 1$ for each $i=1,\ldots,k$ and
the notation $\ones_a\in\R^k$ for the
vector with components $(\ones_a)_a = 1$ and $(\ones_a)_i = 0$ for
$i\neq a$.

%%%%%%%%%%%%%%%
%
%le3.1 #&#
\begin{lemma}\label{signos}
Under hypothesis \textup{(H1)}, if $\vxi\in V_{\un{z}}^{(2)}\setminus\{\vec
{0}\}$, then
$\vxi$ has two nonzero components of opposite signs, $\mu$-almost surely.
\end{lemma}
%
%%%%%%%%%%%%%%%

\begin{pf}
Assume there exits $\vxi\in V_{\un{z}}^{(2)}\setminus\{\vec{0}\}$ with
$\vxi_i\geq0$ for all $i=1,\ldots,k$.
Then, from hypothesis (H1) it follows that there exists $\alpha>0$
such that, for all $\un{z}$,
\[
L_{\un{z}}\vxi\geq\alpha\|\vxi\|\ones.
\]
One may take, for example,
\[
\alpha= \frac{1}{\sqrt{k}} \inf_{z_0,i,j} q(z_0|i)p(j|i)
= \frac{1}{\sqrt{k}} \inf_{\un{z},i,j}(L_{\un{z}})_{i,j}. %
\]
We can apply $L^{[n-1]}_{\shift^{-1}\un{z}}$ to both sides, use
monotonicity and take norms, to obtain
\[
\bigl\| L^{[n]}_{\un{z}}\vxi\bigr\|\geq\alpha\|\vxi\|\bigl\|
L^{[n-1]}_{\shift^{-1}\un{z}}\ones\bigr\|. %
\]
Let $\vec{w}\in V_{\shift^{-1}\un{z}}^{(1)}\setminus V_{\shift
^{-1}\un{z}}^{(2)}$. Then
\[
\bigl\| L^{[n-1]}_{\shift^{-1}\un{z}}\vec{w}\bigr\|\leq\bigl\|
L^{[n-1]}_{\shift^{-1}\un{z}}|
\vec{w}|\bigr\|\leq\|\vec{w}\|\bigl\| L^{[n-1]}_{\shift^{-1}\un
{z}}\ones
\bigr\|\le\frac{\|\vec{w}\|}{\alpha\|\vxi\|}\bigl\|
L^{[n]}_{\un{z}}\vxi\bigr\|.
\]
Therefore
\[
\bigl\| L^{[n]}_{\un{z}}\vxi\bigr\|\geq\frac{\alpha\|
\vxi\|}{\|\vec{w}\|} \bigl\|
L^{[n-1]}_{\shift^{-1}\un
{z}}\vec{w}\bigr\|,
\]
and using Oseledec's theorem we have $\mu$-almost surely that
\[
\lim_{n\to+\infty} \frac1n \log\bigl\| L^{[n]}_{\un{z}}
\vxi\bigr\|\geq\lim_{n\to+\infty} \frac1n \log\bigl\| L^{[n-1]}_{\shift
^{-1}\un{z}}
\vec{w}\bigr\|= \lambda_1,
\]
which contradicts the fact that $\vxi\in V_{\un{z}}^{(2)}\setminus\{
\vec{0}\}$.
\end{pf}

%%%%%%%%%%%%%%
%
%le3.2 #&#
\begin{lemma}\label{codim}
Under hypothesis \textup{(H1)} we have $\operatorname{Codim}(V_{\un{z}}^{(2)}) = 1$,
$\mu$-almost surely.
\end{lemma}
%
%%%%%%%%%%%%%%

\begin{pf}
Assume $\operatorname{Codim}(V_{\un{z}}^{(2)}) \geq2$.
Since any vector $\vec w_{1}$ of norm one in the cone
$\C_k = \{\vec{w}\dvtx\vec{w}>0\}$ does not belong to
$V_{\un{z}}^{(2)}$ (by Lemma \ref{signos}), the vector space $V_{\un
{z}}^{(2)}\oplus\R
\vec w_{1}$ is of codimension at least one,
$\mu$-almost surely. Therefore we can find a vector
$\vec w_{2}$ of norm one
in $\C_k\setminus(V_{\un{z}}^{(2)}\oplus\R\vec w_{1}
)$. Note that
%
%e3.1 #&#
\begin{equation}
\label{v2indep} \inf_{\vec y\in V_{\un{z}}^{(2)}, \gamma} \|\vec
w_{1}-\gamma\vec
w_{2}-\vec y \| > 0
\end{equation}
since otherwise, the minimum is reached at a finite nonzero pair
$(\gamma,\vec
y)$ which would contradict $\vec w_{2}\in\C_k\setminus
(V_{\un{z}}^{(2)}\oplus\R\vec w_{1} )$.
Let $\un{z}$ be a fixed element in $\B^{\Z}$.\vadjust{\goodbreak}

Define
\[
\gamma_n = \max_{i}\frac{ (L_{\un{z}}^{[n]}\vec{w}_1 )_{i}} {
(L_{\un{z}}^{[n]}\vec{w}_2 )_{i}} \quad\mbox{and}\quad
\delta_n = \min_{i}\frac{ (L_{\un{z}}^{[n]}\vec{w}_1 )_{i}} {
(L_{\un{z}}^{[n]}\vec{w}_2 )_{i}}.
\]
Let
\[
\phi= \inf_{\un{z}}\min_{i,j,r,s}
\frac{(L_{\un{z}})_{r,j}
(L_{\un{z}})_{s,i} } {
(L_{\un{z}})_{s,j} (L_{\un{z}})_{r,i}}. %
\]
It follows from hypothesis (H1) that $\phi>0$.
Let
\[
\alpha= \frac{1-\sqrt\phi}{1+\sqrt\phi} < 1. %
\]
From the Birkhoff--Hopf theorem [see, e.g., \citet{cavazos}],
there exists a constant $\beta>0$
such that for all $\un{z}\in\B^\Z$ and all $n$,
%
%e3.2 #&#
\begin{equation}
\label{bir} 1 \leq\frac{\gamma_n}{\delta_n} \leq1+ \beta\alpha^n.
\end{equation}
We now prove that
\[
\frac{\gamma_n}{1+\beta\alpha^n} \leq\delta_{n+1} \leq\gamma_{n+1} \leq
\gamma_n.
\]
To see this observe that $\delta_{n+1}\leq\gamma_{n+1}$ by
definition. We also have by monotonicity of $L_{\un{z}}$
\begin{eqnarray*}
\gamma_{n+1} &=& \max_{i} \frac{(L_{\un{z}}^{[n+1]}\vec{w}_1)_i}{(L_{\un
{z}}^{[n+1]}\vec{w}_2)_i} = \max
_{i}\frac{(L_{\shift^{-n}\un{z}}L_{\un{z}}^{[n]}\vec{w}_1)_i} {
(L_{\un{z}}^{[n+1]}\vec{w}_2)_i}
\\
&\leq& \max_{i}\frac{(L_{\shift^{-n}\un{z}}\gamma_{n} L_{\un
{z}}^{[n]}\vec{w}_2)_i} {
(L_{\un{z}}^{[n+1]}\vec{w}_2)_i} = \gamma_n
\end{eqnarray*}
and also
\[
\delta_{n+1} \ge\delta_{n}= \gamma_n
\frac{\delta
_n}{\gamma_n} \geq\frac{\gamma_n}{1+\beta\alpha^n}.
\]
Since the sequence $(\gamma_{n})$ is decreasing,
there exists $\gamma^*$ and $\beta'>0$ such that
\[
\bigl|\gamma_n - \gamma^*\bigr| \leq\beta' \alpha^n.
\]
On the other hand, it follows immediately from (\ref{bir}) that
for any $i=1,\ldots,k$, we have
\[
- \frac{\gamma_n \beta\alpha^n (L_{\un{z}}^{[n]}\vec
{w}_2)_i}{1+\beta\alpha^n} \leq\bigl(L_{\un{z}}^{[n]}\vec
{w}_1 \bigr)_i-\gamma_n
\bigl(L_{\un{z}}^{[n]}\vec{w}_2 \bigr)_i
\leq0.
\]
Then there exists $\beta''>0$ such that
\[
\frac{\| L^{[n]}_{\un{z}}\vec{w}_1 - \gamma_nL^{[n]}_{\un
{z}}\vec{w}_2\|}{\| L^{[n]}_{\un{z}}\vec{w}_2\|} \leq\beta''
\alpha^n.\vadjust{\goodbreak}
\]
This implies
\begin{eqnarray*}
\frac{\| L^{[n]}_{\un{z}}\vec{w}_1 - \gamma^*L^{[n]}_{\un
{z}}\vec{w}_2\|}{\| L^{[n]}_{\un{z}}\vec{w}_2\|} &\leq& \frac
{|\gamma_n - \gamma^*| \| L^{[n]}_{\un{z}}\vec
{w}_2\|}{\| L^{[n]}_{\un{z}}\vec{w}_2\|} + \frac{\|
L^{[n]}_{\un{z}}\vec{w}_1 - \gamma_nL^{[n]}_{\un{z}}\vec{w}_2\|
}{\| L^{[n]}_{\un{z}}\vec{w}_2\|}
\\
&\leq& \bigl(\beta'+\beta'' \bigr)
\alpha^n.
\end{eqnarray*}
Since $\vec{w}_1$ and $\vec{w}_2$ are linearly independent, we have
$\vec{w}_1-\gamma^*\vec{w}_2 \neq\vec{0}$. This and the previous
inequality imply that
\[
\lim_{n\to+\infty} \frac1n\log\bigl\| L^{[n]}_{\un{z}}
\bigl(\vec{w}_1 - \gamma^*\vec{w}_2 \bigr)\bigr\|\leq
\lambda_1 + \log\alpha< \lambda_1,
\]
then $\vec{w}_1 - \gamma^*\vec{w}_2\in V^{(2)}_{\un{z}}\setminus\{
0\}$,
and this contradicts (\ref{v2indep}).
\end{pf}

The proof of Theorem \ref{upperbound} will follow from the next
proposition.

%%%%%%%%%%%%%%%%%%%%%%%
%
%pr3.3 #&#
\begin{proposition}\label{propupperbound}
Let $(\vpsi_{a})_{a\in\A}$ be a basis of $\R^{k}$ satisfying
$\vpsi_{a}\ge0$ for any~$a$.
Let $\vtheta_{1}>0$ and $\vtheta_{2}>0$ be two vectors in $\R^{k}$.
Then
\[
\limsup_{n\to\infty}\frac{1}{n}\log\biggl\llvert
\frac{\langle\vtheta_1, L^{[n-1]}_{\shift^{-1}\un{z}}\vpsi_a\rangle}{
\langle\vtheta_1, L^{[n-1]}_{\shift^{-1}\un{z}}\sum_{a}\vpsi_a\rangle}
- \frac{
\langle\vtheta_2, L^{[n-1]}_{\shift^{-1}\un{z}}\vpsi_a\rangle}{
\langle\vtheta_2, L^{[n-1]}_{\shift^{-1}\un{z}}\sum_{a}\vpsi_a\rangle}
\biggr\rrvert\le\lambda_{2}-
\lambda_{1}.
\]
\end{proposition}
%
%%%%%%%%%%%%%%%%%%%%%%%

\begin{pf}
By Lemma \ref{signos}, since $(\vpsi_a)_i\ge0$ for any $i=1,\ldots,k$
then $\vpsi_a\notin V^{(2)}_{\un{z}}$, $\mu$-almost surely.
In the same way, from Lemma \ref{signos} we have
\[
\vpsi=\sum_{a\in\A}\vpsi_{a} \in
V^{(1)}_{\un{z}}\setminus V^{(2)}_{\un{z}}.
\]
Note also that since $(\vpsi_{a})_{a\in\A}$ form a basis of
nonnegative vectors, we must have $\vpsi_i>0$ for all $i=1,\ldots,k$.
Therefore, by Lemma \ref{codim} we have that for any $a\in\A$,
%
%e3.3 #&#
\begin{equation}
\label{decomp} \vpsi_a = u_a\vpsi+
\vxi_a,
\end{equation}
where $\vxi_a\in V^{(2)}_{\un{z}}$, $u_a\neq0$, and this
decomposition is unique.
Then
\[
\frac{\langle\vtheta_j, L^{[n-1]}_{\shift^{-1}\un{z}}\vpsi_a\rangle}{
\langle\vtheta_j, L^{[n-1]}_{\shift^{-1}\un{z}}\sum_{a}\vpsi_a\rangle}
= u_a + \frac{\langle\vtheta_j,
L^{[n-1]}_{\shift^{-1}\un{z}}\vxi_a\rangle}{
\langle\vtheta_j, L^{[n-1]}_{\shift^{-1}\un{z}}\vpsi\rangle},\qquad j=1,2.
\]
%
%%%
Define for any $n$ and $\un{z}$
%
%e3.4 #&#
\begin{equation}
\label{gammab} \gamma(n,\un{z}) = \frac{\langle\vtheta_{1},
L^{[n-1]}_{\shift
^{-1}\un{z}}\vpsi\rangle}{\langle\vtheta_2, L^{[n-1]}_{\shift^{-1}\un
{z}}\vpsi\rangle}
\end{equation}
and let
\[
R=\max\biggl\{\sup_{i}\frac{(\vtheta_{1})_i}{(\vtheta_{2})_i}, \sup
_{i}\frac{(\vtheta_{2})_i}{(\vtheta_{1})_i} \biggr\}. %
\]
Then we have
\begin{eqnarray*}
\bigl\langle\vtheta_{1}, L^{[n-1]}_{\shift^{-1}\un{z}}\vpsi\bigr
\rangle&=& \sum_{i=1}^k (
\vtheta_{1})_i \bigl( L^{[n-1]}_{\shift^{-1}\un{z}}\vpsi
\bigr)_i \leq R \sum_{i=1}^k
(\vtheta_2)_i \bigl( L^{[n-1]}_{\shift
^{-1}\un{z}}
\vpsi\bigr)_i
\\
&=& R \bigl\langle\vtheta_2, L^{[n-1]}_{\shift^{-1}\un{z}}\vpsi
\bigr\rangle
\end{eqnarray*}
and similarly
\[
\bigl\langle\vtheta_{2}, L^{[n-1]}_{\shift^{-1}\un{z}}\vpsi\bigr
\rangle\leq R \bigl\langle\vtheta_1, L^{[n-1]}_{\shift^{-1}\un{z}}
\vpsi\bigr\rangle. %
\]
In other words for any $n$ and $\un{z}$,
%
%e3.5 #&#
\begin{equation}
\label{bornegamma} R^{-1} \le\gamma(n,\un{z}) \le R.
\end{equation}
Then
\begin{eqnarray*}
&&\frac{\langle\vtheta_1, L^{[n-1]}_{\shift^{-1}\un{z}}\vpsi_a\rangle}{
\langle\vtheta_1, L^{[n-1]}_{\shift^{-1}\un{z}}\sum_{a}\vpsi_a\rangle}
- \frac{
\langle\vtheta_2, L^{[n-1]}_{\shift^{-1}\un{z}}\vpsi_a\rangle}{
\langle\vtheta_2, L^{[n-1]}_{\shift^{-1}\un{z}}\sum_{a}\vpsi_a\rangle
}
\\
&&\qquad = \bigl( \bigl\langle\vtheta_{1}, L^{[n-1]}_{\shift^{-1}\un{z}}
\vpsi\bigr\rangle\bigr)^{-1} \bigl\langle\vtheta_1-
\gamma(n,\un{z}) \vtheta_2, L^{[n-1]}_{\shift^{-1}\un{z}}
\vxi_a\bigr\rangle.
\end{eqnarray*}
Note that
\[
\bigl| \bigl\langle\vtheta_{1}, L^{[n-1]}_{\shift^{-1}\un{z}}\vpsi\bigr
\rangle\bigr| \geq\frac{1}{\sqrt{k}} \bigl\| L_{\shift^{-1}\un{z}}^{[n-1]}
\vpsi\bigr\|
\inf_{i} \bigl\{ (\vtheta_{1})_i \bigr
\}
\]
and
\begin{eqnarray*}
\bigl| \bigl\langle\vtheta_1-\gamma(n,\un{z})\vtheta_2,
L^{[n-1]}_{\shift^{-1}\un{z}}\vxi_a\bigr\rangle\bigr| &\leq& \bigl\|
\vtheta_1-\gamma(n,\un{z})\vtheta_2 \bigr\|\cdot\bigl\|
L_{\shift^{-1}\un{z}}^{[n-1]} \vxi_a \bigr\|
\\
&\leq& \bigl(\|\vtheta_1\|+ R \|\vtheta_2\|\bigr) \bigl\|
L_{\shift^{-1}\un{z}}^{[n-1]} \vxi_a \bigr\|.
\end{eqnarray*}
Then, using Oseledec's theorem the result follows.
\end{pf}

%%%%%%%%%%%%%%%%%%%%%%%%%%
%
\begin{pf*}{Proof of Theorem \ref{upperbound}}
%%%%%%%%%%%%%%%%%%%%%%%%%%
We observe that
\begin{eqnarray*}
&&
\sum_{b\in\A}\P\bigl(X_0=a,Z_{-n+1}^{-1}=z_{-n+1}^{-1},X_{-n}=b
\bigr) \rho(b)
\\
&&\qquad=\sum_{b\in\A}\sum_{x_{-n+1}^{-1}\in\A^{n-1}}
p(x_{-n+1}|b)\rho(b)q(z_{-1}|x_{-1})
p(a|x_{-1}) \\
&&\hspace*{99pt}{}\times\prod_{l=1}^{n-2}
q(z_{-l-1}|x_{-l-1})p(x_{-l}|x_{-l-1})
\\
&&\qquad= \bigl\langle\vtheta_\rho, L^{[n-1]}_{\shift^{-1}\un{z}}
\vpsi_a \bigr\rangle,
\end{eqnarray*}
where $\vtheta_\rho,\vpsi_a\in\R^k$ are given by
%
%e3.6 #&#
\begin{equation}\label{lepsi1}
(\vtheta_\rho)_i = \sum_{b\in\A}
\rho(b) p(i|b) \quad\mbox{and}\quad \vpsi_a = \ones_a.
\end{equation}
Therefore,
\[
\nu^{[n]}_{\un{z}, \rho}(a) = \frac{\langle\vtheta_\rho,
L^{[n-1]}_{\shift^{-1}\un{z}}\vpsi_a\rangle}{\langle\vtheta_\rho,
L^{[n-1]}_{\shift^{-1}\un{z}}\sum_{a}\vpsi_a\rangle},
\]
and the first statement of Theorem \ref{upperbound} follows from
Proposition \ref{propupperbound} since the conditions on $(\psi
_a)_{a\in\A}$, $\vtheta_{\rho}$
and $\vtheta_{\rho'}$ can be
immediately verified using hypothesis (H1).

The second part follows similarly by noting that
\[
\sum_{c\in\B}\P\bigl(X_0=a,Z_{-n+1}^{-1}=z_{-n+1}^{-1},Z_{-n}=c
\bigr) \eta(c) = \bigl\langle\vtheta_\eta, L^{[n-1]}_{\shift^{-1}\un{z}}
\vpsi_a \bigr\rangle,
\]
where
%
%e3.7 #&#
\begin{equation}
\label{lepsi2} (\vtheta_\eta)_i = \sum
_{c\in\B}\sum_{x\in\A}p(i|x) q(c|x)
\eta(c) \quad\mbox{and}\quad \vpsi_a = \ones_a.
\end{equation}
This finishes the proof of Theorem \ref{upperbound}.
\end{pf*}

Before proceeding with the proof of Theorem \ref{lowerbound} we will
prove a useful lemma.

%%%%%%%%%%%%%%%
%
%le3.4 #&#
\begin{lemma}\label{base1}
Let $(\vpsi_{a})_{a\in\A}$ be a basis of $\R^{k}$ such that
$\vpsi_{a}\ge0$ for all $a$.
Let $\tilde\Omega$ be a set of full $\mu$-measure where the
Oseledec's theorem holds. Then for any $\un{z}\in\bar
\Omega$, there exists a
symbol $a=a(\un{z})\in\A$ such that $\vxi_a\in V^{(2)}_{\un
{z}}\setminus
V^{(3)}_{\un{z}}$, where $\vxi_a$ is the unique vector in
$V^{(2)}_{\un{z}}$ satisfying
\[
\vpsi_a = u_a\sum_{b\in\A}
\vpsi_{b} + \vxi_a
\]
for some real number $u_{a}$.
\end{lemma}
%
%%%%%%%%%%%%%%%

\begin{pf}
Assume
$\vxi_a\in V^{(3)}_{\un{z}}$ for all $a$.
Then, as
$\operatorname{Codim}(V^{(3)}_{\un{z}})\geq2$, the set $\{\vpsi_a\}$
generates a sub-space of co-dimension 1. This contradicts
the fact that the set of vectors $\{\vpsi_a\dvtx a\in\A\}$
forms a basis of $\R^k$.
\end{pf}

The proof of Theorem \ref{lowerbound} will follow from the next
proposition.
%
%pr3.5 #&#
\begin{proposition}\label{proplowerbound}
Let $(\vpsi_{a})_{a\in\A}$ be a basis of $\R^{k}$ satisfying
$\vpsi_{a}\ge0$ for any $a$.
Let $(\vtheta_{j})_{j\in\A}$ be another basis of $\R^{k}$ such that
$\vtheta_{j}>0$.\vadjust{\goodbreak}
Then for $\mu$-almost every $\un{z}$ there exist $a\in\A$ and two
indices $r, s\in\{1,\ldots,k\}$ such that
%
%e3.8 #&#
\begin{equation}\qquad
\limsup_{n\to\infty}\frac{1}{n}\log\biggl\llvert
\frac{\langle\vtheta_r, L^{[n-1]}_{\shift^{-1}\un{z}}\vpsi_a\rangle}{
\langle\vtheta_r, L^{[n-1]}_{\shift^{-1}\un{z}}\sum_{a}\vpsi_a\rangle}
- \frac{
\langle\vtheta_s, L^{[n-1]}_{\shift^{-1}\un{z}}\vpsi_a\rangle}{
\langle\vtheta_s, L^{[n-1]}_{\shift^{-1}\un{z}}\sum_{a}\vpsi_a\rangle}
\biggr\rrvert\ge\lambda_{2}-
\lambda_{1}.
\end{equation}
\end{proposition}
\begin{pf}
Let $\tilde\Omega$ be a set of full $\mu$-measure where the
Oseledec's theorem holds.
Applying Lemma \ref{base1}, for any $\un{z}\in\tilde\Omega$ we
find a symbol
$a=a(\un{z})\in\A$ such that
\[
\vpsi_a = u_a\sum_{b\in\A}
\vpsi_{b} + \vxi_a %
\]
with $\vxi_a\in
V^{(2)}_{\un{z}}\setminus V^{(3)}_{\un{z}}$.
Let
%
%e3.9 #&#
\begin{equation}
\tilde\xi_a (n,\un{z}) = \frac{L^{[n-1]}_{\shift^{-1}\un
{z}}\vxi_a}{\|L^{[n-1]}_{\shift^{-1}\un{z}}\vxi_a\|} \in
V^{(2)}_{\shift^{-n}\un{z}}.
\end{equation}
We now show that there exist $r$ and $s$ such that
\[
\limsup_{n\to\infty}{ \bigl|\bigl\langle\tilde\theta_r(n,
\un{z}) - \tilde\theta_s(n,\un{z}), \tilde\xi_a(n,
\un{z})\bigr\rangle\bigr|} > 0,
\]
where the vectors $\tilde\theta_j(n,\un{z})$ are defined by
\[
\tilde\theta_j(n,\un{z})=\frac{\langle\vtheta_{1},
L^{[n-1]}_{\shift^{-1}\un{z}}\sum_{a}\vpsi_{a}\rangle}{\langle\vtheta_j,
L^{[n-1]}_{\shift^{-1}\un{z}}\sum_{a}\vpsi_{a}\rangle} \vtheta_{j}.
\]
Assume this is not the case, namely that for any $r$ and $s$,
%
%e3.10 #&#
\begin{equation}
\label{lafalso} \lim_{n\to\infty}{ \bigl|\bigl\langle\tilde
\theta_r(n, \un{z}) -\tilde\theta_s(n,\un{z}), \tilde
\xi_a(n, \un{z})\bigr\rangle\bigr|} = 0.
\end{equation}
Choose for any $n$ (and fixed $\un{z}$) a normalized vector
$\vec{f}(n,\un{z})$ orthogonal
to $V^{(2)}_{\shift^{-n}\un{z}}$. Such a vector exists by Lemma
\ref{codim}. Note that for any $j$, $n$ and $\un{z}$, we have
\[
0 < {R^{-1} \min_{m}}\|\vec\theta_{m}\|
\leq\bigl\| \tilde\theta_j(n,\un{z})\bigr\| \leq {R \max_m}
\|\vec\theta_{m}\|, %
\]
where
\[
R=\sup_{j,m}\sup_{i}\frac{(\vtheta_{j})_i}{(\vtheta_{m})_i}.
\]
This implies that the vectors $(\vec{f}(n,\un{z}), \tilde\xi
_a(n,\un{z}),
\tilde\theta_1(n,\un{z}),\ldots, \tilde\theta_k(n,\un{z}))$ belong
to a compact subset of $\R^{k+2}$. Therefore, we can find a
subsequence $(n_{j})$ of integers such that
\begin{eqnarray*}
&&
\lim_{j\to\infty} \bigl(\vec{f}(n_{j},\un{z}), \tilde\xi
_a(n_{j},\un{z}), \tilde\theta_1(n_{j},
\un{z}),\ldots, \tilde\theta_k(n_{j},\un{z}) \bigr)
\\
&&\qquad= \bigl(\bar f(\un{z}), \bar\xi_a(\un{z}), \bar
\theta_1(\un{z}),\ldots, \bar\theta_k(\un{z}) \bigr).
\end{eqnarray*}
The vectors $\bar f(\un{z})$ and $\bar\xi_a(\un{z})$ have norm one,
and the
vectors $\bar\theta_j(\un{z})$ have nonnegative components and satisfy
\[
0 < {R^{-1} \min_m}\|\vec\theta_{m}\|
\leq\bigl\| \bar\theta_j(\un{z})\bigr\| \leq {R \max_m}\|
\vec\theta_{m}\|. %
\]
We have also for any $r$ and $s$
\[
\bigl\langle\bar\theta_r(\un{z}) -\bar\theta_s(
\un{z}), \bar\xi_a(\un{z})\bigr\rangle= 0. %
\]
We now show that the set of vectors $ \{\bar\theta_m(\un{z})
\}$
is a basis of $\R^{k}$. This follows from
\begin{eqnarray*}
\bigl|\det\bigl(\bar\theta_{1}(\un{z}),\ldots,\bar\theta_{k}(
\un{z}) \bigr) \bigr| &=& \bigl|\det(\vec\theta_{1},\ldots,\vec
\theta_{k} ) \bigr| \lim_{j\to\infty} \prod
_{m=1}^{k} \biggl\llvert\frac{\langle\vtheta_{1},
L^{[n_{j}-1]}_{\shift^{-1}\un{z}}\sum_{a}\vpsi_{a}\rangle}{\langle
\vtheta_m,
L^{[n_{j}-1]}_{\shift^{-1}\un{z}}\sum_{a}\vpsi_{a}\rangle} \biggr
\rrvert
\\
&\geq& R^{-k} \bigl|\det(\vec\theta_{1},\ldots,\vec
\theta_{k} ) \bigr| > 0.
\end{eqnarray*}
Let
\[
\zeta(n,\un{z})=\frac{1}{k}\sum_{m=1}^{k}
\tilde\theta_{m}(n,\un{z})
\]
and
\[
\bar\zeta(\un{z})=\lim_{j\to\infty} \zeta(n_{j},\un{z})=
\frac{1}{k}\sum_{m=1}^{k} \bar
\theta_{m}(\un{z}). %
\]
We now observe that since all the components of the vector
$\zeta(n,\un{z})$ are strictly positive, and since by Lemma \ref{signos}
any vector in $ V^{(2)}_{\shift^{-n}\un{z}}$ has two components of
opposite sign, we get
\begin{eqnarray*}
\frac{|\langle\vec{f}(n,\un{z}), \zeta(n,\un{z})\rangle| } {
\|\zeta(n,\un{z})\|} &=& \inf_{\vec y\in V^{(2)}_{\shift^{-n}\un{z}}} \biggl\|
\frac{\zeta(n,\un{z})}{\|\zeta(n,\un{z})\|}-\vec y
\biggr\|
\\
&\geq&\min_{i} \biggl\{\frac{(\zeta(n,\un{z}))_{i}}{\|\zeta(n,\un
{z})\|} \biggr\} \geq
\frac{1}{k R^{2}} \frac{\min_{m,i}(\vec\theta_{m})_i}{\max_{m,i}(\vec
\theta_{m})_i} > 0.
\end{eqnarray*}
Taking the limit we get
\[
\frac{|\langle\bar f(\un{z}), \bar\zeta(\un{z})\rangle|} {
\|\bar\zeta(\un{z})\|} \geq\frac{1}{k R^{2}} \frac{\min_{m,i}(\vec
\theta_{m})_i}{\max_{m,i}(\vec\theta_{m})_i} > 0. %
\]
We now define the orthogonal projection $\mathcal{P}$ on the orthogonal
$\bar f^{\perp}$
of $\bar f$ parallel to~$\bar\zeta$, namely for any vector $v$
\[
\mathcal Pv=v-\bar\zeta\frac{\langle\bar f, v\rangle} {
\langle\bar f, \bar\zeta\rangle}. %
\]
We claim that the vectors $ (\mathcal{P} (\bar
\theta_{m}(\un{z})-\bar
\theta_{m+1}(\un{z}) ) )_{m=1,\ldots,k-1}$ form a basis of
$\bar
f^{\perp}$. Indeed, if this is not true, there exist real numbers
$\alpha_{1},\ldots,\alpha_{k-1}$, with at least one nonzero, such that
\[
\sum_{m=1}^{k-1}\alpha_{m}
\mathcal{P} \bigl(\bar\theta_{m}(\un{z})-\bar\theta_{m+1}(
\un{z}) \bigr)=0. %
\]
In other words, there exists a number $\alpha$ such that
\[
\sum_{m=1}^{k-1}\alpha_{m}
\bigl(\bar\theta_{m}(\un{z})-\bar\theta_{m+1}(\un{z})
\bigr)=\alpha\bar\zeta. %
\]
But this is impossible since the vectors $ (\bar
\theta_{m}(\un{z})-\bar
\theta_{m+1}(\un{z}) )_{m=1,\ldots,k-1}$ and $\bar\zeta(\un{z})$
form a basis of $\R^{k}$. Since
\[
\bigl\langle\bar f(\un{z}), \bar\xi_{a}(\un{z}) \bigr\rangle= \lim
_{j\to\infty} \bigl\langle f(n_{j},\un{z}), \tilde
\xi_{a}(n_{j},\un{z}) \bigr\rangle= 0, %
\]
we obtain that the normalized vector $\bar\xi_{a}(\un{z})$ would be
orthogonal to the basis $ (\mathcal{P} (\bar
\theta_{m}(\un{z})-\bar
\theta_{m+1}(\un{z}) ) )_{m=1,\ldots,k-1}$ of $\bar
f^{\perp}$
which is a contradiction with (\ref{lafalso}). In other words, there
exists $a=a(\un{z})$, $r=r(\un{z})$ and $s=s(\un{z})$ such that
\[
\limsup_{n\to\infty}{ \bigl|\bigl\langle\tilde\theta_r(n,
\un{z}) - \tilde\theta_s(n,\un{z}), \tilde\xi_a(n,
\un{z})\bigr\rangle\bigr|} > 0.
\]
By Schwarz's inequality we have
\begin{eqnarray*}
&& \biggl\llvert\frac{\langle\vtheta_r, L^{[n-1]}_{\shift^{-1}\un
{z}}\vpsi_a\rangle}{
\langle\vtheta_r, L^{[n-1]}_{\shift^{-1}\un{z}}\sum_{a}\vpsi_a\rangle}
- \frac{
\langle\vtheta_s, L^{[n-1]}_{\shift^{-1}\un{z}}\vpsi_a\rangle}{
\langle\vtheta_s, L^{[n-1]}_{\shift^{-1}\un{z}}\sum_{a}\vpsi_a\rangle}
\biggr\rrvert
\\
&&\qquad= \frac{\| L^{[n-1]}_{\shift^{-1}\un{z}}\vxi_a\|} {
|\langle\vtheta_{1},
L^{[n-1]}_{\shift^{-1}\un{z}}\sum_{a}\vpsi_a\rangle|} \bigl|\bigl\langle
\tilde\theta_r(n,\un{z})-\tilde
\theta_s(n,\un{z}), \tilde\xi_{a}(n,\un{z})\bigr\rangle\bigr|
\\
&&\qquad\geq\frac{\| L^{[n-1]}_{\shift^{-1}\un{z}}\vxi_a\|} {
\|\vtheta_{1}\| \| L^{[n-1]}_{\shift^{-1}\un{z}}\sum_{a}\vpsi
_a\|} \bigl|\bigl\langle\tilde\theta_r(n,\un{z})-
\tilde\theta_s(n,\un{z}), \tilde\xi_{a}(n,\un{z})\bigr
\rangle\bigr|.
\end{eqnarray*}
Therefore, for this choice of $a(\un{z})\in\A$, $r(\un{z})$ and
$s(\un{z})$, we have
\[
\limsup_{n\to\infty}\frac{1}{n}\log\biggl\llvert
\frac{\langle\vtheta_r, L^{[n-1]}_{\shift^{-1}\un{z}}\vpsi_a\rangle}{
\langle\vtheta_r, L^{[n-1]}_{\shift^{-1}\un{z}}\sum_{a}\vpsi_a\rangle
}-\frac{
\langle\vtheta_s, L^{[n-1]}_{\shift^{-1}\un{z}}\vpsi_a\rangle}{
\langle\vtheta_s, L^{[n-1]}_{\shift^{-1}\un{z}}\sum_{a}\vpsi_a\rangle}
\biggr\rrvert\ge\lambda_{2}-
\lambda_{1}. %
\]
\upqed
\end{pf}

%%%%%%%%%%%%%%%%%%%%%%%%%%
%
\begin{pf*}{Proof of Theorem \ref{lowerbound}}
%%%%%%%%%%%%%%%%%%%%%%%%%%
As in the proof of Theorem \ref{upperbound} we take for any $a\in\A
$ the vector $\vpsi_{a}=\ones_a$. We also take
for any $b\in\B$ the vector $\vtheta_b$ in $\R^k$ as the vector
$\vtheta_{\rho}$ in (\ref{lepsi1}) with $\rho$ the Dirac measure
concentrated on $b$, that is, $(\vtheta_b)_i = p(i|b)$.
Under (H2), these definitions verify the hypotheses of Proposition \ref
{proplowerbound} and the first part of Theorem \ref{lowerbound}
follows.

For the second part, we define for any $c\in\B$ the vector $(\vtheta
_c)$ as the vector $(\vtheta_{\eta})$
in (\ref{lepsi2}) with $\eta$ the Dirac measure concentrated on $c$,
that is,
\[
(\vtheta_c)_i =\sum_{x\in\A}
p(i|x) q(c|x). %
\]
It follows from (H2) and (H3) that
we can choose $c_{1},\ldots,c_{k}$ such that
$(\vtheta_{c_{j}})_{1\leq j\leq k}$ is a basis of $\R^k$. The result
follows again
from Proposition \ref{proplowerbound}.
\end{pf*}

%%%%%%%%%%%%%%%%%%%%%%%%
%
\begin{pf*}{Proof of Corollary \ref{cormain}}
%%%%%%%%%%%%%%%%%%%%%%%%
The upper bound follows by noting that for all $\un{z}\in\B^\Z$, for
any $e\in\B$ and for any measures $\rho$ and $\rho'$ on $\A$, we have
%
%e3.11 #&#
\begin{equation}
\label{equiv} \tilde\nu^{[n]}_{\un{z}, \rho}(e)-\tilde
\nu^{[n]}_{\un{z},
\rho'}(e)= \sum_{x_{0}\in\A}
q(e|x_{0}) \bigl(\nu^{[n]}_{\un{z},
\rho}(x_{0})
- \nu^{[n]}_{\un{z}, \rho'}(x_{0}) \bigr),
\end{equation}
and applying Theorem \ref{upperbound}, the second upper bound follows
similarly.

We now prove that the upper bound is reached for almost all
$\un{z}\in\B^\Z$.

By (H3), as $\operatorname{rank}(q) = k$ there exists symbols
$e_1,\ldots,e_k\in\B$ such that the matrix
$M\in\R^{k\times k}$ with elements $M_{i,j} = q(e_i|j)$ is invertible.
For $b, c\in\A$, denote by
$U^{[n]}_{b,c,\un{z}}$ and $V^{[n]}_{b,c,\un{z}}$ the vectors in $\R
^k$ with elements
$(U^{[n]}_{b,c,\un{z}})_i = \tilde\nu^{[n]}_{\un{z}, b}(e_{i})
-\tilde\nu^{[n]}_{\un{z}, c}(e_{i})$ and
$(V^{[n]}_{b,c,\un{z}})_i = \nu^{[n]}_{\un{z}, b}(i)
-\nu^{[n]}_{\un{z}, c}(i)$. By (\ref{equiv}) we have
\[
U^{[n]}_{b,c,\un{z}} = M V^{[n]}_{b,c,\un{z}}
\]
and as $M$ is invertible
\[
V^{[n]}_{b,c,\un{z}} = M^{-1} U^{[n]}_{b,c,\un{z}}.
\]
Then, for all $a, b, c\in\A$
\begin{eqnarray*}
\bigl| \nu^{[n]}_{\un{z}, b}(a) -\nu^{[n]}_{\un{z}, c}(a) \bigr|
&\leq&\bigl\|V^{[n]}_{b,c,\un{z}} \bigr\| \leq\bigl\| M^{-1}\bigr\| \bigl\|
U^{[n]}_{b,c,\un{z}}\bigr\| %
\\
&\leq&\sqrt{k} \bigl\| M^{-1}\bigr\| \max_{i} \bigl\{ \bigl| \tilde
\nu^{[n]}_{\un{z}, b}(e_{i}) -\tilde\nu^{[n]}_{\un{z}, c}(e_{i})
\bigr| \bigr\}. %
\end{eqnarray*}
Applying the logarithm on both sides, dividing by $n$ an taking limits,
we have that for all $\un{z}$ on a set of positive measure, for all
$a, b, c\in\A$, and for all $e\in\B$
\[
\limsup_{n\to\infty}\frac{1}{n} \log\bigl| \nu^{[n]}_{\un{z}, b}(a)
-\nu^{[n]}_{\un{z}, c}(a) \bigr|\le\max_{e\in\B} \limsup
_{n\to\infty}\frac{1}{n} \log\bigl| \tilde\nu^{[n]}_{\un{z}, b}(e)
-\tilde\nu^{[n]}_{\un{z}, c}(e) \bigr|, %
\]
and the third part of Corollary \ref{cormain} follows from
Theorem \ref{lowerbound}.
The last part follows by the same arguments.
\end{pf*}

%%%%%%%%%%%%%%%%%%%%%%
%
\begin{pf*}{Proof of Proposition \ref{diff}}
%%%%%%%%%%%%%%%%%%%%%%
It is well known that the sequence of Lyapunov exponents
satisfy
\[
\lambda_1 + m_2\lambda_2+\cdots+m_s\lambda_s = \E_\mu\bigl[ {\log}|{\det
L_{\un{z}}}| \bigr],
\]
where the numbers $m_i$ denote the multiplicity of $\lambda_i$, namely
$\operatorname{dim}(V^{(j)}_{\un{z}})=m_{j}+\cdots+m_s $; see
\citet{ledrappier,katok1995}. In particular, $1+m_2+\cdots+m_s =k$.
Let $E = \E_\mu[ {\log}|{\det L_{\un{z}}}| ]$. Then we have
\[
E \leq\lambda_1 + (k-1)\lambda_2
\]
and
\[
\lambda_2 -\lambda_1 \geq\frac{E}{k-1} -
\frac{k}{k-1}\lambda_1.
\]
Note that by Lemma \ref{signos}, for almost all $\un{z}$ we have
\[
\lambda_1 = \lim_{n\to\infty}\frac1n \log
\bigl\|L_{\un{z}}^{[n]}\ones\bigr\| \leq\lim_{n\to\infty}\frac1n
\log\|\ones\| = 0.
\]
Moreover,
\[
\det L_{\un{z}} = \Biggl( \prod_{i=1}^k
q(z_0|i) \Biggr) \det(p).
\]
Therefore,
\begin{eqnarray*}
\lambda_2 -\lambda_1 &\geq& \frac{1}{k-1}\log\bigl|
\det(p)\bigr| + \frac{1}{k-1} \sum_{i=1}^k
\E_\mu\bigl[\log q( \cdot|i) \bigr]
\\
&\geq& \frac{1}{k-1}\log\bigl|\det(p)\bigr| - \frac{k}{k-1}\log\Gamma.
\end{eqnarray*}
\upqed
\end{pf*}

Before proving Theorem \ref{tv}, we prove a lemma in linear algebra
which will be useful for the
proof. This lemma is probably well known but we could not find a
reference. We give the proof here for the convenience of the reader.

%le3.6 #&#
\begin{lemma}\label{alglin}
Let $e_{1},\ldots,e_{k}$ be a basis of $\R^{k}$ and assume that all
the vectors $\|e_{j}\|$ have norm one. Then, for any vector $v\in\R^{k}$
of norm one, we have
\[
\sup_{1\le j\le k} \bigl|\langle v, e_{j}\rangle\bigr| \ge
\frac{1}{ k^{3/2} (k-1)! \det A}, %
\]
where $A$ is a matrix mapping the basis $(e_{i})$ to
an orthonormal basis.
\end{lemma}
\begin{pf}
Let
\[
\delta=\sup_{1\le j\le k} \bigl|\langle v, e_{j}\rangle\bigr|.\vadjust{\goodbreak}
\]
Let $(f_{j})$ be an orthonormal basis of $\R^{k}$. Let $A$ be the
matrix mapping the basis $(e_{j})$ to the basis $(f_{j})$, namely
$A e_{j}=f_{j}$ for $1\le j\le k$. We have
\[
\langle v, e_{j}\rangle= \bigl\langle v, A^{-1}f_{j}
\bigr\rangle= \bigl\langle{A^{-1}}^{\mathrm{t}} v, f_{j}
\bigr\rangle. %
\]
Therefore,
\[
\bigl\| {A^{-1}}^{\mathrm{t}} v \bigr\|\le\delta\sqrt k
\]
and
\[
1=\|v\|= \bigl\|{A}^{\mathrm{t}}{A^{-1}}^{\mathrm{t}}v \bigr\| \le
\bigl\|{A}^{\mathrm{t}} \bigr\|\delta\sqrt k. %
\]
On the other hand, since $(e_{j})_{\ell}=A^{-1}_{j,\ell}$ we have
$|A^{-1}_{j,\ell}|\le1$ for $1\le j\le k$ and $1\le\ell\le k$. This
implies by the well-known formula expressing the elements of an
inverse matrix in terms of minors and determinant that
for any $1\le j\le k$ and $1\le\ell\le k$
\[
|A_{j,\ell} |\le\frac{(k-1)!}{ |{\det A^{-1}}
|}=(k-1)! \det A. %
\]
Therefore $\|A^{\mathrm{t}}
\|=\|A\|\le k (k-1)! \det A$ (the Hilbert--Schmidt norm), and
we finally get
\[
\delta\ge\frac{1}{ k^{3/2} (k-1)! |{\det A }|}.
\]
\upqed
\end{pf}

Theorem \ref{tv} will be a consequence of the following proposition.

%pr3.7 #&#
\begin{proposition}\label{proptv}
Assume hypotheses\vspace*{1pt} \textup{(H1)--(H2)} hold.
Let $(\vpsi_{a})_{a\in\A}$ be a basis of $\R^{k}$ satisfying\vspace*{2pt}
$\vpsi_{a}\ge0$ for any $a$.
Let $\vtheta_{1}>0$ and $\vtheta_{2}>0$ be two vectors in $\R^{k}$
with $\vtheta_{1}$ independent of $\vtheta_{2}$.
Then
\begin{eqnarray*}
&&\liminf_{n\to\infty}\frac{1}{n}\log\sum
_{a\in\A} \biggl\llvert\frac{\langle\vtheta_1, L^{[n-1]}_{\shift
^{-1}\un{z}}\vpsi_a\rangle}{
\langle\vtheta_1, L^{[n-1]}_{\shift^{-1}\un{z}}\sum_{a}\vpsi_a\rangle}
- \frac{
\langle\vtheta_2, L^{[n-1]}_{\shift^{-1}\un{z}}\vpsi_a\rangle}{
\langle\vtheta_2, L^{[n-1]}_{\shift^{-1}\un{z}}\sum_{a}\vpsi_a\rangle}
\biggr\rrvert
\\
&&\qquad \ge\lambda_{s}-\lambda_{1}+\sum
_{i=2}^{s}m_{i} (\lambda_{i}-
\lambda_{2} ) \ge2 \bigl({\log}|{\det p}|-k\log\Gamma\bigr).
\end{eqnarray*}
\end{proposition}
\begin{pf}
As in the proof of Proposition \ref{propupperbound} let
\[
\vpsi=\sum_{a\in\A}\vpsi_{a} \quad\mbox{and}\quad
\gamma(n,\un{z}) = \frac{\langle\vtheta_{1},
L^{[n-1]}_{\shift^{-1}\un{z}}\vpsi\rangle}{\langle\vtheta_2,
L^{[n-1]}_{\shift^{-1}\un{z}}\vpsi\rangle}. %
\]
We also define the vector $\vec{\eta}(n,\un{z})=\vtheta_1 -\gamma
(n,\un{z}) \vtheta_2$ that satisfies
\[
\bigl\langle\vec{\eta}(n,\un{z}), L^{[n-1]}_{\shift^{-1}\un{z}}\vpsi
\bigr
\rangle=0.
\]

For any $a\in\A$, we denote as before by $\vxi_{a}$ the unique vector
in $ V_{\un{z}}^{(2)}$ such that
$\vpsi_{a}=u_{a} \vpsi+\vxi_{a}$, with $u_{a}$ a real number.

Let
%
%e3.12 #&#
\begin{equation}
\label{xit} \tilde\xi_a (n,\un{z}) = \frac{L^{[n-1]}_{\shift^{-1}\un
{z}}\vxi_a}{
\|L^{[n-1]}_{\shift^{-1}\un{z}}\vxi_a\|} \in
V^{(2)}_{\shift
^{-n}\un{z}}.
\end{equation}

Let $a_{2},\ldots,a_{k}$ be any given collection of $k-1$ different
elements of $\A$. Then, for any $\un{z}\in\Omega$, the set of
vectors $\vpsi,\vxi_{a_{2}},\ldots,\vxi_{a_{k}}$ form a basis of
$\R^{k}$ (since $\vpsi\notin V_{\un{z}}^{(2)}$ by Lemma
\ref{signos}).

By the hypotheses (H1)--(H2) we have that for any $\un{z}$, $\det
(L_{\un{z}})\neq0$, and therefore
for any integer $n$, $\det(L^{[n]}_{\un{z}})\neq0$.
This implies that the collection of vectors
%
%e3.13 #&#
\begin{equation}
\label{setvectors} \biggl\{ \frac{L^{[n-1]}_{\shift^{-1}\un{z}}\vpsi}{
\|L^{[n-1]}_{\shift^{-1}\un{z}}\vpsi\|}, \tilde
\xi_{a_{2}}(n,\un{z}),\ldots,\tilde\xi_{a_{k}}(n,\un{z}) \biggr\}
\end{equation}
is also a basis of $\R^{k}$ and
\begin{eqnarray*}
&&\sum_{j=2}^{k} \biggl\llvert
\frac{\langle\vtheta_1, L^{[n-1]}_{\shift^{-1}\un{z}}\vpsi
_{a_{j}}\rangle}{
\langle\vtheta_1, L^{[n-1]}_{\shift^{-1}\un{z}}\vpsi\rangle} - \frac{
\langle\vtheta_2, L^{[n-1]}_{\shift^{-1}\un{z}}\vpsi_{a_{j}}\rangle}{
\langle\vtheta_2, L^{[n-1]}_{\shift^{-1}\un{z}}\vpsi\rangle} \biggr
\rrvert
\\
&&\qquad \geq\sum_{j=2}^{k} \frac{\| L^{[n-1]}_{\shift^{-1}\un{z}}\vxi
_{a_{j}}\|} {
\|\vtheta_{1}\| \| L^{[n-1]}_{\shift^{-1}\un{z}}\vpsi\|}
\bigl|\bigl\langle\vtheta_1-\gamma(n,\un{z}) \vtheta_2,
\tilde\xi_{a_{j}}(n,\un{z})\bigr\rangle\bigr|
\\
&&\qquad= \sum_{j=2}^{k} \frac{\| L^{[n-1]}_{\shift^{-1}\un{z}}\vxi_{a_{j}}\|} {
\|\vtheta_{1}\| \| L^{[n-1]}_{\shift^{-1}\un{z}}\vpsi\|} \bigl|
\bigl\langle\vec\eta(n,\un{z}), \tilde\xi_{a_{j}}(n,\un{z})\bigr\rangle\bigr|.
\end{eqnarray*}
We now apply Lemma \ref{alglin}
and obtain
\begin{eqnarray*}
&&\sum_{j=2}^{k} \biggl\llvert
\frac{\langle\vtheta_1, L^{[n-1]}_{\shift^{-1}\un{z}}\vpsi
_{a_{j}}\rangle}{
\langle\vtheta_1, L^{[n-1]}_{\shift^{-1}\un{z}}\vpsi\rangle} - \frac{
\langle\vtheta_2, L^{[n-1]}_{\shift^{-1}\un{z}}\vpsi_{a_{j}}\rangle}{
\langle\vtheta_2, L^{[n-1]}_{\shift^{-1}\un{z}}\vpsi\rangle} \biggr
\rrvert
\\
&&\qquad \geq\frac{\inf_{a\in\A}\| L^{[n-1]}_{\shift^{-1}\un{z}}\vxi_{a}\|} {
\|\vtheta_{1}\| \| L^{[n-1]}_{\shift^{-1}\un{z}}\vpsi\|} \frac{ \|\vec
\eta(n,\un{z}) \|}{ k^{3/2} (k-1)!
\llvert
{\det M}
\rrvert},
\end{eqnarray*}
where $M$ is the matrix formed by the vectors in (\ref{setvectors}).
We now observe that
\[
\bigl\|\vec\eta(n,\un{z}) \bigr\|= \bigl\|\vtheta_{1}- \gamma(n,\un{z})
\vtheta_{2} \bigr\|\ge\inf_{\alpha} \| \vtheta_{1}-
\alpha\vtheta_{2} \|=\sqrt{ \|\vtheta_{1} \|
^{2}-\frac{
\langle\vtheta_{1}, \vtheta_{2}\rangle^{2} } {
\|\vtheta_{2} \|^{2}}} >0, %
\]
since $\vtheta_{1}$ is independent of $\vtheta_{2}$.
We also have
\begin{eqnarray*}
\det M & = & \frac{\det
(L^{[n-1]}_{\shift^{-1}\un{z}}\vpsi,
L^{[n-1]}_{\shift^{-1}\un{z}}\vxi_{a_{2}},\ldots,
L^{[n-1]}_{\shift^{-1}\un{z}}\vxi_{a_{k}} )} {
\|L^{[n-1]}_{\shift^{-1}\un{z}}\vpsi\| \prod_{j=2}^{k}
\|L^{[n-1]}_{\shift^{-1}\un{z}}\vec
\xi_{a_{j}}\|}
\\
&=& \frac{\det(L^{[n-1]}_{\shift^{-1}\un{z}} )} {
{\|L^{[n-1]}_{\shift^{-1}\un{z}}\vpsi\| \prod_{j=2}^{k}
\|L^{[n-1]}_{\shift^{-1}\un{z}}\vec
\xi_{a_{j}}\|}} \det(\vpsi, \vxi_{a_{2}},\ldots,
\vxi_{a_{k}} ).
\end{eqnarray*}
We observe that, for any $a\in A$, by Oseledec's theorem we have
$\mu$-almost surely
\[
\liminf_{n\to\infty} \frac{1}{n}\log\bigl\|L^{[n-1]}_{\shift^{-1}\un
{z}}
\vec\xi_{a}\bigr\| \ge\lambda_{s}. %
\]
Therefore [see, e.g., \citet{ledrappier,katok1995}], since
\[
\lim_{n\to\infty} \frac{1}{n}\log\bigl|\det\bigl(L^{[n-1]}_{\shift^{-1}\un{z}}
\bigr) \bigr| = \sum_{j=1}^{s}m_{j}
\lambda_{j}, %
\]
we get
\begin{eqnarray*}
&&\liminf_{n\to\infty} \frac{1}{n}\log\sum
_{a\in\A} \biggl\llvert\frac{\langle\vtheta_1, L^{[n-1]}_{\shift
^{-1}\un{z}}\vpsi_{a}\rangle}{
\langle\vtheta_1, L^{[n-1]}_{\shift^{-1}\un{z}}\vpsi\rangle} - \frac{
\langle\vtheta_2, L^{[n-1]}_{\shift^{-1}\un{z}}\vpsi_{a}\rangle}{
\langle\vtheta_2, L^{[n-1]}_{\shift^{-1}\un{z}}\vpsi\rangle}
\biggr\rrvert
\\
&&\qquad\ge\liminf_{n\to\infty} \frac{1}{n}\log\sum
_{j=2}^{k} \biggl\llvert\frac{\langle\vtheta_1, L^{[n-1]}_{\shift
^{-1}\un{z}}\vpsi
_{a_{j}}\rangle}{
\langle\vtheta_1, L^{[n-1]}_{\shift^{-1}\un{z}}\vpsi\rangle} -
\frac{
\langle\vtheta_2, L^{[n-1]}_{\shift^{-1}\un{z}}\vpsi_{a_{j}}\rangle}{
\langle\vtheta_2, L^{[n-1]}_{\shift^{-1}\un{z}}\vpsi\rangle} \biggr
\rrvert
\\
&&\qquad \ge\lambda_{s}-\lambda_{1}+\sum
_{j=1}^{s}m_{j}\lambda_{j}
-(k-1)\lambda_{2}-\lambda_{1}
\\
&&\qquad = \lambda_{s}-\lambda_{1} +\sum
_{j=2}^{s}m_{j} (\lambda_{j}-
\lambda_{2} ),
\end{eqnarray*}
which is the first part of the lower bound. We also have
\[
\bigl|\det\bigl(L^{[n-1]}_{\shift^{-1}\un{z}} \bigr) \bigr| = |{\det p}|^{n-1}
\prod_{j=1}^{n} \Biggl(\prod
_{\ell=1}^{k} q (z_{-j},l ) \Biggr) \ge|{\det
p}|^{n-1} \Gamma^{-nk}. %
\]
Therefore
\[
\sum_{j=1}^{s}m_{j}
\lambda_{j} \ge{\log}|{\det p}|-k\log\Gamma. %
\]
Since all the Lyapunov exponents are nonpositive, we get
\begin{eqnarray*}
\lambda_{s}-\lambda_{1} +\sum_{j=2}^{s}m_{j}
(\lambda_{j}-\lambda_{2} ) &\ge& \lambda_{s}+\sum
_{j=2}^{s}m_{j}
\lambda_{j}
\\
&\ge& 2\sum_{j=1}^{s}m_{j}
\lambda_{j} \\
&\ge& {2\log}|{\det p}|-2k\log\Gamma.
\end{eqnarray*}
\upqed
\end{pf}

%%%%%%%%%%%%%%%%%%%%%
%
\begin{pf*}{Proof of Theorem \ref{tv}}
%%%%%%%%%%%%%%%%%%%%%
The result follows immediately from Proposition~\ref{proptv} using
the same choices for the vectors $(\vpsi_{a})_{a\in\A}$ and
$(\vtheta_b)_{b\in\B}$ as in the proof
of Theorem \ref{lowerbound}.
\end{pf*}

%%%%%%%%%%%%%%%%%%%%%%%%%%%%%%%%%%
%s4 #&#
\section{Perturbed processes over a binary alphabet}\label{binary}
%%%%%%%%%%%%%%%%%%%%%%%%%%%%%%%%%%

Consider the Mar\-kov chain $(X_t)_{t\in\Z}$ over the alphabet $\A=\{
0,1\}$
with matrix of transition probabilities given by
\[
P = \pmatrix{ p_0 & 1-p_0
\cr p_1 & 1-p_1},
\]
where we assume $p_0\neq p_1$ and
\[
0<\beta=\min\{p_{0}, p_{1}, 1-p_{0},1-p_{1}
\}. %
\]
The quantities $p(j|i)$ are given by
$p(j|i)=P_{i,j}$.

Consider also the process $(Z_t)_{t\in\Z}$ over the alphabet $\B=\{
0,1\}$ with output matrix
$q_{\varepsilon}(j|i)= \P(Z_0=j\given
X_0=i)=(1-\varepsilon) \mathbh{1}_{\{i=j\}}+\varepsilon
\mathbh{1}_{\{i\neq j\}}$. From now on we will
assume $\varepsilon\in(0,1)\setminus\{1/2\}$. Then, as $z_0\in\{0,1\}$,
\[
L_{\un{z},\varepsilon} = \pmatrix{ \bigl[z_0\varepsilon+(1-z_0)
(1-\varepsilon) \bigr] p_0 & \bigl[z_0\varepsilon
+(1-z_0) (1-\varepsilon) \bigr](1-p_0)
\vspace*{2pt}\cr
\bigl[z_0(1-\varepsilon) +(1-z_0)\varepsilon
\bigr]p_1 & \bigl[z_0(1-\varepsilon) +(1-z_0)
\varepsilon\bigr](1-p_1)}.
\]
We have the
following equality:
\[
\lambda_1+\lambda_2 = \E_\mu\bigl[ {\log}|{\det
L_{ \cdot,\varepsilon}}| \bigr];
\]
see, for example, \citet{ledrappier} or \citet{katok1995} for a proof. Therefore
%
%e4.1 #&#
\begin{eqnarray}\label{ledet}
\lambda_1+\lambda_2 & = & \P(Z_0=0) \log
\bigl( (1-\varepsilon)\varepsilon\bigl|p_0(1-p_1)-p_1(1-p_0)\bigr|
\bigr)
\nonumber
\\
&&{}+ \P(Z_0=1) \log\bigl( \varepsilon(1-\varepsilon)\bigl|p_0(1-p_1)-p_1(1-p_0)\bigr|
\bigr)
\\
& = & \log\varepsilon+ \log(1-\varepsilon) + \log\bigl|\det(P)\bigr|.
\nonumber
\end{eqnarray}
From the above expression for $L_{\un{z},\varepsilon}$ we have
\[
L_{\un{z},\varepsilon} = M_{z_0} + \varepsilon A_{z_0},
\]
where
\[
M_{z_0} = \pmatrix{ (1-z_0) p_0 &
(1-z_0) (1-p_0)
\cr
z_0 p_1 &
z_0 (1-p_1)}
\]
and
\[
A_{z_0} = (2z_0-1) \pmatrix{ p_0 &
(1-p_0)
\cr
-p_1 & -(1-p_1)}.
\]
For $z_{0}\in\{0,1\}$ define the vectors
\[
\vec{e}_{z_{0}} = \pmatrix{ 1-z_{0}
\cr
z_{0}}
\quad\mbox{and}\quad \vec{f}_{z_{0}} = \pmatrix{ z_{0}
\cr
1-z_{0}}.
\]
These vectors have norm 1 and satisfy
\[
M_{z_{0}} \vec e_{z_{1}} =\rho_{0}(z_{0},z_{1})
\vec e_{z_{0}} \quad\mbox{and}\quad M_{z_{0}}^{t} \vec
f_{z_{0}}=\vec0, %
\]
where
\[
\rho_{0}(z_{0},z_{1})=(1-z_{1})
\bigl(p_{0}(1-z_{0})+p_1z_0 \bigr)
+z_1 \bigl( (1-p_{0}) (1-z_{0})+(1-p_{1})z_0
\bigr), %
\]
since $z_{0}$ and $z_{1}$ equal zero or one.

We recall that a distance $d$ can be defined on $\Omega$ as follows.
For $\un{z}$ and $\un{z}'$ in $\Omega$, let
\[
\tilde d \bigl(\un{z},\un{z}' \bigr) = \inf\bigl\{|i|,
z_{i}\neq z'_{i} \bigr\}. %
\]
Then
\[
d \bigl(\un{z},\un{z}' \bigr) = e^{-\tilde d(\un{z},\un{z}')}. %
\]
We refer to \citet{bowen} for details, in particular $\Omega$ equipped
with this distance is a compact metric space. We now
prove the following result.
%%%%%%%
%
%le4.1 #&#
\begin{lemma}\label{projectif}
There exist two constants $\varepsilon_{0}>0$ and $D>0$
and two continuous functions
$\rho(\varepsilon,\un{z})$ and $h(\varepsilon,\un{z})$
such that for any
$\varepsilon\in[0,\varepsilon_{0}]$,
the vectors
\[
\vec{g}(\varepsilon,\un{z}) = \vec{e}_{z_1} + \varepsilon h(\varepsilon,\un{z})
\vec{f}_{z_1}
\]
satisfy
\[
L_{\un{z},\varepsilon} \vec{g}(\varepsilon,\un{z}) = \rho(\varepsilon,\un{z})
\vec{g}
\bigl(\varepsilon,\shift^{-1}\un{z} \bigr).
\]
Moreover, there is a constant $U>1$ such that for any
$\varepsilon\in[0,\varepsilon_{0}]$, any $n$ and any $\un{z}\in\Omega$,
\[
\bigl\|\vec{g}(\varepsilon,\un{z})-\vec e_{z_{1}}\bigr\| \leq U \varepsilon,\qquad \bigl|\rho(
\varepsilon,\un{z})-\langle M_{z_1}\vec{e}_{z_2},
\vec{e}_{z_1}\rangle\bigr| \leq U \varepsilon%
\]
and
\[
U^{-1} \varepsilon\vec1 \leq\vec{g}(\varepsilon,\un{z}) \leq U \vec1.
\]
\end{lemma}
\begin{pf}
The equation for $\vec g$ is equivalent to
%
%e4.2 #&#
\begin{equation}
\label{eqproj} L_{\shift\un{z},\varepsilon} \vec{g}(\varepsilon,\shift\un{z})
= \rho(\varepsilon,
\shift\un{z}) \vec{g}(\varepsilon,\un{z}).
\end{equation}
Note that
\[
\vec{g}(\varepsilon,\shift\un{z}) = \vec{e}_{z_2} + \varepsilon h(\varepsilon,
\shift\un{z}) \vec{f}_{z_2} \quad\mbox{and}\quad L_{\shift\un{z},\varepsilon}=M_{z_{1}}+
\varepsilon A_{z_{1}}.
\]
Taking the scalar product of both terms in equation (\ref{eqproj})
with $\vec{e}_{z_1}$ and $\vec{f}_{z_1}$ we get
%
%e4.3 #&#
\begin{eqnarray}\label{lerho}
\rho(\varepsilon,\shift\un{z}) &=& \langle M_{z_1}\vec{e}_{z_2},
\vec{e}_{z_1}\rangle+ \varepsilon h(\varepsilon,\shift\un{z})\langle
M_{z_1}\vec{f}_{z_2}, \vec{e}_{z_1}\rangle
\nonumber\\[-8pt]\\[-8pt]
&&{}+ \varepsilon\langle A_{z_1}\vec{e}_{z_2},
\vec{e}_{z_1}\rangle+ \varepsilon^2 h(\varepsilon,\shift\un{z})
\langle A_{z_1}\vec{f}_{z_2},\vec{e}_{z_1} \rangle\nonumber
\end{eqnarray}
and since $M_{z_{1}}^{t}\vec{f}_{z_{1}}=0$ and $\langle\vec
f_{z_{1}},\vec e_{z_{1}}\rangle=0$,
\[
\rho(\varepsilon,\shift\un{z}) h(\varepsilon,\un{z}) = \langle A_{z_1}
\vec{e}_{z_2},\vec{f}_{z_1}\rangle+\varepsilon h(\varepsilon,\shift
\un{z})\langle A_{z_1}\vec{f}_{z_2},\vec{f}_{z_1}
\rangle.
\]
We denote by $\mathcal{D}$ the Banach space of continuous functions on
$[0,\varepsilon_0]\times\Omega$ equipped with the sup norm. On the ball $B_{D}$
of radius $D=4\beta^{-1}$ centered at the origin in $\mathcal{D}$ we
define a
transformation $\mathscr{T}$ given by
%
%e4.4 #&#
\begin{equation}
\label{leT} \mathscr{T}(h) (\varepsilon,\un{z}) = \frac{u_{1}(\varepsilon,\un
{z})+\varepsilon u_{2}
(\varepsilon,\un{z}) h(\varepsilon,\shift\un{z})} {
u_{3}(\varepsilon,\un{z})+\varepsilon u_{4}
(\varepsilon,\un{z}) h(\varepsilon,\shift\un{z})},
\end{equation}
where
\begin{eqnarray*}
u_{1}(\varepsilon,\un{z}) &=& \langle A_{z_1}
\vec{e}_{z_2},\vec{f}_{z_1}\rangle,\qquad u_{2}(
\varepsilon,\un{z})=\langle A_{z_1}\vec{f}_{z_2},
\vec{f}_{z_1}\rangle, %
\\
u_{3}(\varepsilon,\un{z})&=&\langle M_{z_1}
\vec{e}_{z_2},\vec{e}_{z_1}\rangle+ \varepsilon\langle
A_{z_1}\vec{e}_{z_2},\vec{e}_{z_1}\rangle
\end{eqnarray*}
and
\[
u_{4}(\varepsilon,\un{z})= \langle M_{z_1}
\vec{f}_{z_2}, \vec{e}_{z_1}\rangle+ \varepsilon\langle
A_{z_1}\vec{f}_{z_2},\vec{e}_{z_1} \rangle.
\]
Direct computation shows that for all $(\varepsilon,\un{z})\in
[0,\varepsilon_0]\times\Omega$ we have
\begin{eqnarray*}
\beta\leq\bigl|u_{1}(\varepsilon,\un{z})\bigr| &\leq&1,\qquad \beta\leq
\bigl|u_{2}(\varepsilon,\un{z})\bigr| \leq1,\qquad \beta\leq\frac{u_{1}(\varepsilon,\un
{z})}{u_{3}(\varepsilon,\un{z})}\le
\beta^{-1}, %
\\
\beta- \varepsilon\leq\bigl|u_{3}(\varepsilon,\un{z})\bigr| &\leq&1+\varepsilon,\qquad \beta-
\varepsilon\leq\bigl|u_{4}(\varepsilon,\un{z})\bigr| \leq1+\varepsilon.
\end{eqnarray*}

We first prove that $\mathscr{T}$ maps $B_{D}$ into itself. Indeed for
$h\in B_{D}$, since $D=4\beta^{-1}$ there exists $\varepsilon'_{0}>0$ small
enough such that for any $\varepsilon\in[0,\varepsilon'_{0}]$,
\[
\bigl|\mathscr{T}(h) (\varepsilon,\un{z}) \bigr| \le\frac{1+\varepsilon D }{\beta
-\varepsilon-\varepsilon D(1+\varepsilon)} \leq D. %
\]
We leave to the reader the proof that $\mathscr{T}(h)$ is a continuous
function of $\varepsilon$ and $\un{z}$.
We now prove that $\mathscr{T}$ is a contraction on
$B_{D}$. For $h$ and $h'$ in $B_{D}$, since $D=4\beta^{-1}$ there exists
$\varepsilon_{0}>0$ small
enough, and smaller than $\varepsilon'_{0}$,
such that for any $\varepsilon\in[0,\varepsilon_{0}]$
we have
\begin{eqnarray*}
&&
\bigl|\mathscr{T}(h) (\varepsilon,\un{z}) -\mathscr{T} \bigl(h' \bigr) (
\varepsilon,\un{z}) \bigr|
\\
&&\qquad= \varepsilon\biggl| \frac{u_{1}(\varepsilon,\un{z})u_{4}(\varepsilon,\un{z})-
u_{2}(\varepsilon,\un{z})u_{3}(\varepsilon,\un{z})}{
(u_{3}(\varepsilon,\un{z})+\varepsilon u_{4}
(\varepsilon,\un{z}) h(\varepsilon,\shift\un{z}) )
(u_{3}(\varepsilon,\un{z})+\varepsilon u_{4}
(\varepsilon,\un{z}) h'(\varepsilon,\shift\un{z}) )
} \biggr| \\
&&\qquad\quad{}\times\bigl|h(\varepsilon,\un{z})-h'(\varepsilon,
\un{z}) \bigr|
\\
&&\qquad\leq\varepsilon\frac{4}{ (
\beta-\varepsilon-\varepsilon D(1+\varepsilon) )^{2}} \bigl|h(\varepsilon,\un
{z})-h'(\varepsilon,
\un{z}) \bigr| \leq\frac
{1}{2} \bigl|h(\varepsilon,\un{z})-h'(\varepsilon,
\un{z}) \bigr|.
\end{eqnarray*}
By the contraction mapping principle [see, e.g., \citet{dieudonne}],
the map $\mathscr{T}$ has a unique fixed point $h$ in $B_{D}$. It
follows at once that the vectors
\[
\vec{g}(\varepsilon,\un{z}) = \vec{e}_{z_1} + \varepsilon h(\varepsilon,\un{z})
\vec{f}_{z_1} %
\]
satisfy equation (\ref{eqproj}). The estimate on
$\vec{g}(\varepsilon,\un{z})$ follows immediately from the fact that
$h\in
B_{D}$, and from (\ref{leT}),
\[
h(\varepsilon,\un{z})= \frac{u_{1}(\varepsilon,\un{z})}{u_{3}(\varepsilon,\un
{z})}+\mathcal{O}(\varepsilon). %
\]
The estimate on $\rho(\varepsilon,\un{z})$
follows from (\ref{lerho}).
\end{pf}

\begin{remark*}
An easy improvement of the above proof allows to show that $\rho$ and
$h$ depend analytically on $\varepsilon$ in a small (complex)
neighborhood of
$0$.
\end{remark*}

By the estimate on $\vec{g}(\varepsilon,\un{z})$ of the previous lemma and
Lemma \ref{signos} applied to the vector $\vec1$, we have
$\mu$-almost surely
\[
\lim_{n\to\infty}\frac{1}{n}\log\bigl\| L_{\shift^{-1}\un{z}}^{[n-1]}
\vec{g}(\varepsilon,\un{z})\bigr\| = \lim_{n\to\infty}\frac{1}{n}\log\bigl\|
L_{\shift^{-1}\un{z}}^{[n-1]} \vec1\bigr\| = \lambda_{1}. %
\]
On the other hand from Lemma \ref{projectif} it follows that
\[
\log\bigl\| L_{\shift^{-1}\un{z}}^{[n-1]} \vec{g}(\varepsilon,\un{z})\bigr\| = \sum
_{j=0}^{n}\log\rho\bigl(\varepsilon,
\shift^{-j} \un{z} \bigr)+\bigl\| \vec g \bigl(\varepsilon,\shift^{-n}
\un{z} \bigr) \bigr\|. %
\]
Using again the estimate on $\vec{g}(\varepsilon,\un{z})$ from
Lemma \ref{projectif},
the Birkhoff ergodic theorem [\citet{Krengel}] and the
ergodicity of $\mu$, we have
\[
\lambda_{1}=\int\log\rho(\varepsilon,\un{z}) \,d \mu(\un{z}). %
\]
The first Lyapunov exponent $\lambda_{1}$ is equal to
$H$ the entropy of the process $(Z_t)_{t\in\Z}$
and this entropy has an expansion in terms of $\varepsilon$; see
\citet{JSS}.
Therefore
\[
H=H_{0}+\mathcal{O}(\varepsilon),
\]
where $H_{0}$ is the entropy of the Markov chain
$(X_t)_{t\in\Z}$.\vadjust{\goodbreak}

The following theorem is an immediate consequence of the above estimates.
%%%%%%%%%%%%%%%%%%%%%%%%%%%%%%%%%%%%%%
%
%th4.2 #&#
\begin{theorem}
If $p_{0}\neq p_{1}$, $\min\{p_{0}, p_{1} 1-p_{0},1-p_{1}\}>0$
and $\varepsilon>0$ is small enough, we have $\mu$-almost surely
\[
\limsup_{n\to+\infty}\frac1n \log\bigl\llvert\nu^{[n]}_{\un{z}, b}(a)-
\nu^{[n]}_{\un{z}, c}(a) \bigr\rrvert\le\log\varepsilon+\log\bigl|\det(P)\bigr|
-2 H_{0}+\mathcal{O}(\varepsilon). %
\]
Moreover, for $\mu$-almost all $\un{z}$
there is a triplet $(a,b,c)$ (which may depend on~$\un{z}$)
where the equality holds.
\end{theorem}
%
%%%%%%%%%%%%%%%%%%%%%%%%%%%%%%%%%%%
%
\begin{pf}
It is easy to verify that hypotheses (H1)--(H2) are satisfied.
We
therefore apply Theorems \ref{upperbound} and \ref{lowerbound}.
The
result follows from (\ref{ledet})
and the above estimate on $\lambda_{1}$.
\end{pf}

As $\lambda_1$ and $\lambda_2$ are fixed, the above estimate also
applies to the asymptotic rate of exponential
loss of memory of the measures
$\sigma^{[n]}_{\un{z}, \eta}$, $\tilde\nu^{[n]}_{\un{z},
\rho}$ and
$\tilde\sigma^{[n]}_{\un{z}, \eta}$.

\section*{Acknowledgments}

The authors thank A. Galves for interesting discussions about the
subject, and an anonymous referee for useful suggestions. This work is
part of FAPESP's project \emph{NeuroMat} (2011/51350-6) and CNPq's
project \emph{Probabilistic modeling of brain activity}
(480108/2012-9). F. Leonardi is also thankful to the Brazilian--French agreement
for cooperation in mathematics and USP-COFECUB project
(2009.1.820.45.8).

%suskaldyti doi

% imsref loaded by lrinkeviciute, 2013-09-16 10:14:00

\printaddresses

\end{document}